\documentclass[oneside,11pt]{amsart}
\usepackage[T1]{fontenc}
\usepackage[latin9]{inputenc}
\usepackage{amsthm}
\usepackage{amssymb}
\usepackage{mathtools}

\newcommand{\N}{\mathbb{N}}
\newcommand{\Z}{\mathbb{Z}}

\newcommand{\D}{\mathbb{D}}

\newcommand{\K}{\mathbb{K}}
\newcommand{\ad}{\mathrm{ad}\,}
\newcommand{\Aff}{\mathrm{Aff}}
\newcommand{\Alt}{\mathbb{A}}

\newcommand{\End}{\mathrm{End}}
\newcommand{\id}{\mathrm{id}}
\newcommand{\Inn}{\mathrm{Inn}}
\newcommand{\NA}{\mathcal{B}}
\newcommand{\Og}{g^G}
\newcommand{\Oh}{h^G}

\newcommand{\supp}{\mathrm{supp}\,}
\newcommand{\Sym}{\mathbb{S}}
\newcommand{\trid}{\triangleright}
\newcommand{\tril}{\triangleleft}

\makeatletter
\def\imod#1{\allowbreak\mkern0mu({\operator@font mod}\,#1)}
\makeatother

\makeatletter
\numberwithin{equation}{section}
\numberwithin{figure}{section}
\numberwithin{table}{section}
\theoremstyle{plain}
\newtheorem{thm}{Theorem}[section]
\newtheorem{lem}[thm]{Lemma}
\newtheorem{cor}[thm]{Corollary}
\newtheorem{pro}[thm]{Proposition}
\newtheorem{notation}[thm]{Notation}
\newtheorem{defn}[thm]{Definition}

\theoremstyle{remark}
\newtheorem{rem}[thm]{Remark}
\newtheorem{exa}[thm]{Example}

\makeatother

\begin{document}

\title[Nichols algebras with finite root system of rank two]{Nichols algebras over groups\\
with finite root system of rank two II}
\author{I. Heckenberger}
\author{L. Vendramin}
\address{Philipps-Universit\"at Marburg\\ 
FB Mathematik und Informatik \\
Hans-Meerwein-Stra\ss e\\
35032 Marburg, Germany}
\email{heckenberger@mathematik.uni-marburg.de}
\email{lvendramin@dm.uba.ar}
\thanks{
Istv\'an Heckenberger was supported by German Research Foundation via a
Heisenberg professorship.  Leandro Vendramin was supported by Conicet and the
Alexander von Humboldt Foundation.
}

\begin{abstract}
	We classify all non-abelian groups $G$ such that there exists a pair $(V,W)$
	of absolutely simple Yetter-Drinfeld modules over $G$ such that the Nichols
	algebra of the direct sum of $V$ and $W$ is finite-dimensional under two
	assumptions: the square of the braiding between $V$ and $W$ is not the
	identity, and $G$ is generated by the support of $V$ and $W$.  As a
	corollary, we prove that the dimensions of such $V$ and $W$ are at most six.
	As a tool we use the Weyl groupoid of $(V,W)$.
\end{abstract}

\maketitle
\setcounter{tocdepth}{1}

\section*{Introduction}

In the theory of Hopf algebras, deep structure results were achieved
since the introduction of the Lifting Method of Andruskiewitsch and
Schneider \cite{MR1659895}. The aim of the method is to classify
(finite-dimensional) pointed Hopf algebras. The idea of it is to
generalize Lusztig's approach to quantum groups \cite{MR2759715}.

The Lifting Method is based on the understanding of the structure theory of
certain braided Hopf algebras which are known as Nichols algebras. In Lusztig's
setting this is the algebra $\mathbf{f}$, also known as
$U_q(\mathfrak{n}_+)$.  Motivated by the first classification
results of finite-dimensional Nichols algebras of diagonal type
\cite{MR1780094}\cite{MR1632802}, a complete solution was obtained by the first
author \cite{MR2462836}. The tool for the latter classification was the Weyl
groupoid and the root system of a Nichols algebra of diagonal type, which was
discovered in \cite{MR2207786} using the theory of Lyndon words and PBW bases
\cite{MR1763385}. The Weyl groupoid was also used by Angiono to determine the
defining relations of finite-dimensional Nichols algebras of diagonal type
\cite{Ang}. These results have far reaching consequences in the theory of Hopf
algebras such as the classification of finite-dimensional pointed Hopf algebras
with abelian coradical of order coprime to $210$ \cite{MR2630042}, and the
proof of the Andruskiewitsch-Schneider conjecture for pointed Hopf algebras
with abelian coradical \cite{Ang_AS}.

In order to understand the structure of Nichols algebras of non-diagonal type,
the Weyl groupoid of a Nichols algebra of diagonal type was generalized further
in several successive papers such as \cite{MR2766176}, \cite{MR2390080},
\cite{MR2734956} and \cite{MR2732989}.  The first applications of this
generalization were powerful enough to study pointed Hopf algebras in some
cases where the coradical is a finite simple group \cite{MR2786171,MR2745542}.
The difficulties towards extensions of these results and the scientific
curiosity ask for a better understanding of finite-dimensional Nichols algebras
of semisimple Yetter-Drinfeld modules over arbitrary groups.

In \cite{MR2732989}, H.-J.~Schneider and the first author introduced a method
to study the Weyl groupoid of a Nichols algebra over a Hopf algebra with
invertible antipode. The main achievement of the paper was a description of
$(\ad V)^n(W)$ for two Yetter-Drinfeld modules $V,W$ in terms of the braiding.
Also, a family $(\Gamma_n)_{n\ge 2}$ of groups was introduced as candidates
admitting finite-dimensional Nichols algebras, and the finite-dimensional
Nichols algebras over $\Gamma_2$ with finite root system of rank two were
determined.

Roughly speaking, in this paper we prove that any non-abelian group $G$ having
a finite-dimensional Nichols algebra with an irreducible finite root system of
rank two has to be an epimorphic image of $\Gamma_2$, $\Gamma_3$, $\Gamma_4$ or
another group $T$.  As a corollary, we obtain that the dimension of the
subspace of primitive elements of such a Nichols algebra has dimension at most
$12$. For the precise statement we refer to Theorem~\ref{thm:main} and to
Corollary~\ref{cor:main}.  These claims are expected to become very useful in
different ways.  For example, the study of Nichols algebras of tuples of
irreducible Yetter-Drinfeld modules requires usually a good understanding of
the rank two case.  Further, our results combined with the methods in
\cite{MR2786171,MR2745542} can be used to obtain strong restrictions on the
support of an irreducible Yetter-Drinfeld module with finite-dimensional
Nichols algebra over a group.

In order to obtain more precise claims on Nichols algebras over $G$, one has to
perform detailed calculations about $(\ad V)^m(W)$ and $(\ad W)^m(V)$, $m\ge
1$, as in \cite[\S 4]{MR2732989}. These calculations lead then to the
classification of finite-dimensional Nichols algebras with finite root system
of rank two, see \cite{examples, rank2}.

Our method is based on the Weyl groupoid. Let $V,W$ be absolutely simple 
Yetter-Drinfeld modules over $G$ such that the pair $(V,W)$ admits all
reflections and the Weyl groupoid $\mathcal{W}(V,W)$ is finite.  By
\cite[Thm.~3.12, Prop.~3.23]{MR2766176}, this is the case if the
Nichols algebra of $V\oplus W$ is finite-dimensional.  We prove that there
exists an object of $\mathcal{W}(V,W)$ which has a Cartan matrix of finite
type.  Thus we have to analyze the consequences of $(\ad V)^2(W)=0$, $(\ad
W)^4(V)=0$. We obtain restrictions regarding decomposability, size and further
information on $\supp V$ and $\supp W$.
In particular, Theorem~\ref{thm:quandles_and_groups} tells that
for a pair $(V,W)$ of Yetter-Drinfeld modules over $G$ such that $(\ad
V)(W)\not=0$, $(\ad V)^2(W)=0$ and $(\ad W)^4(V)=0$ it is necessary that $\supp
V\cup \supp W$ is isomorphic to one of five quandles, all of size at most six.
Our results are based on
Proposition~\ref{pro:degrees} claiming the non-vanishing of $(\ad V)^{m+1}(W)$
under some assumptions on the structure of $\supp V$ and $\supp W$.  It is an
interesting fact that for this proposition and for many of its consequences we
do not need to assume that $V$ and $W$ have finite support or that their
supports are conjugacy classes. Therefore Proposition~\ref{pro:degrees} and its
consequences can also be used to deal with Nichols algebras of arbitrary tuples
of irreducible Yetter-Drinfeld modules over groups.

The structure of the paper is as follows. First we recall some facts on groups
with abelian centralizers, quandles and their enveloping groups in
Sections~\ref{section:preliminaries} and \ref{section:groups}.  In
Section~\ref{section:Weylgroupoid} we prove with
Corollary~\ref{cor:cartan_type} that connected Weyl groupoids of rank two
admitting a finite irreducible root system have an object with a Cartan matrix
of finite type. Section~\ref{section:Nichols} is devoted to the study of
Nichols algebras over groups. After discussing some technicalities,
we formulate our main results,
Theorem~\ref{thm:main} and Corollary~\ref{cor:main}.
In Section~\ref{section:proof} we
give a step-by-step proof of Theorem~\ref{thm:quandles_and_groups}.

\section{Preliminaries}
\label{section:preliminaries}

\subsection{Groups with abelian centralizers}

Recall from \cite{MR0360813} that a group has \emph{abelian centralizers} if
the centralizer of every non-central element is abelian. The following
definition goes back to Hall \cite{MR0003389}.

\begin{defn}
	\label{defn:isoclinic}
	Let $G$ and $H$ be two groups. We say that $G$ is \emph{isoclinic} to $H$ if
	there exist isomorphisms $\zeta:G/Z(G)\to H/Z(H)$ and $\eta:[G,G]\to [H,H]$
	such that if $g_1,g_2\in G$, $h_1,h_2\in H$, and $\zeta(g_iZ(G))=h_iZ(H)$ for
	$i=1,2$, then $\eta[g_1,g_2]=[h_1,h_2]$.  In this case we write $G\sim H$. 
\end{defn}

It is clear that the relation of isoclinism is an equivalence relation.  The
following lemma is due to Hall \cite[page 134]{MR0003389}.

\begin{lem}
	\label{lem:Hall}
	Let $G$ be a group and $K\vartriangleleft G$. The following hold.
	\begin{enumerate}
		\item\label{it:Hall_1} $G/K\sim G/(K\cap [G,G])$. 
		\item\label{it:Hall_2} If $K\subseteq[G,G]$ and $G\sim H$ for some group
      $H$ via the maps $\zeta $ and $\eta $, then $\eta (K)\vartriangleleft H$
      and	$G/K\sim H/\eta (K)$.
		\end{enumerate}
\end{lem}

The following lemma was proved in \cite[Lemma 3.4]{MR0360813}. For completeness
we give a proof in the context of this paper. 

\begin{lem}
	\label{lem:abelian_centralizers}
	Let $G$ and $H$ be groups and assume $G\sim H$. If $G$ has abelian
	centralizers, then $H$ has abelian centralizers.
\end{lem}

\begin{proof}
	Let $h\in H\setminus Z(H)$ and let $h_1,h_2\in H^h$. Since $G\sim H$, there
	exist $g,g_1,g_2\in G$ such that $\zeta(gZ(G))=hZ(H)$ and
	$\zeta(g_iZ(G))=h_iZ(H)$ for $i=1,2$. Further $g\not\in Z(G)$ since $h\not\in
	Z(H)$. Then $1=[h,h_i]=\eta[g,g_i]$ and hence $g_i\in G^g$. Therefore
	$1=\eta[g_1,g_2]=[h_1,h_2]$ and $H^h$ is abelian. 
\end{proof}

\subsection{Quandles} 

Recall that a \textit{quandle} is a non-empty set $X$ with a binary operation
$\trid$ such that the map $\varphi_i:X\to X$, $j\mapsto i\trid j$, is
bijective for all $i\in X$, $i\trid(j\trid k)=(i\trid j)\trid(i\trid k)$ for
all $i,j,k\in X$, and $i\trid i=i$ for all $i\in X$. The bijectivity of
$\varphi _i$ can be expressed by the existence of a map $\tril :X\times X\to
X$ such that $(i\trid j)\tril i=j=i\trid (j\tril i)$ for all $i,j\in X$. Then
\begin{align}
  k\trid (i\tril j)=(k\trid i)\tril (k\trid j),\quad 
  (i\tril j)\tril k=(i\tril k)\tril (j\tril k)
  \label{eq:tril}
\end{align}
for all $i,j,k\in X$.
A \emph{crossed set} is a
quandle $X$ such that for all $i,j\in X$, $i\trid j=j$ implies $j\trid i=i$.
Unions of conjugacy classes of a group with the binary operation of conjugation
are examples of crossed sets. 

\begin{notation}\label{not:class} \
	\begin{enumerate}
		\item To describe a finite quandle $X$ we may assume that $X=\{1,\dots,n\}$
			for some $n\in\N$ and then write
			$X:\varphi_1\;\varphi_2\;\cdots\;\varphi_n$ to denote the quandle structure on $X$
			given by $\varphi_1,\dots,\varphi_n$.
		\item Let $G$ be a group and $g\in G$. The quandle structure on the
			conjugacy class of $g$ in $G$ will be denoted by $g^G$.	
		\end{enumerate}
\end{notation}

The \textit{inner group} of a quandle $X$ is the group
$\Inn(X)=\langle\varphi_i:i\in X\rangle$ We say that a quandle $X$ is
\textit{indecomposable} if the inner group $\Inn(X)$ acts transitively on $X$.
Also, $X$ is \textit{decomposable} if it is not indecomposable. 

\begin{rem} 
	\label{rem:smallcrossedsets}
	Crossed sets of size at most three are well-known. If $X$ is a crossed set
	and $1\leq|X|\leq2$ then $X$ is trivial (or commutative), that is,
  $i\trid j=j$ for all
	$i,j\in X$. If $|X|=3$ and $X$ is non-trivial, then $i\trid j=k$ for all
  pairwise different elements $i,j,k\in X$. Hence $X\simeq (12)^{\Sym_3}$.
\end{rem}

\begin{rem}
	\label{rem:uptosix}
    Using the classification of transitive groups of small degree,
    indecomposable quandles of small size were classified (up to isomorphism)
    in \cite{MR2926571}. The list of indecomposable quandles of size $\leq6$
    is:
\begin{align*}
  \{1\} &: \id\\
(12)^{\mathbb{S}_{3}}&:(23)\;(13)\;(12)\\
(123)^{\mathbb{A}_{4}}&:(243)\;(134)\;(142)\;(123)\\
\Aff(5,2)&:(2354)\;(1534)\;(1452)\;(1325)\;(1243)\\
\Aff(5,3)&:(2453)\;(1435)\;(1254)\;(1523)\;(1342)\\
\Aff(5,4)&:(25)(34)\;(13)(45)\;(15)(24)\;(12)(35)\;(14)(23)\\
(12)^{\mathbb{S}_{4}}&:(23)(56)\;(13)(45)\;(12)(46)\;(25)(36)\;(16)(24)\;(15)(34)\\
(1234)^{\mathbb{S}_{4}}&:(2436)\;(1654)\;(1456)\;(1253)\;(2634)\;(1352)
\end{align*}
\end{rem}

Let $X$ be a quandle and let $G_X$ denote its enveloping group 
\[
	G_X=\langle x_i\,|\,i\in X\rangle / (x_ix_j=x_{i\trid j}x_i\text{ for all $i,j\in X$}).
\]
This group is $\Z$-graded with $\deg(x_i)=1$ for all $i\in X$. 

\begin{rem}[Universal property]
	\label{rem:universal_property}
	For any group $G$ and any map $f:X\to G$ satisfying $f(x\triangleright
	y)=f(x)f(y)f(x)^{-1}$ there exists a unique group homomorphism $g:G_X\to G$
	such that $f=g\circ\partial$, where $\partial:X\to G_X$, $i\mapsto x_i$, see
	for example \cite[Lemma 1.6]{MR1994219}.
\end{rem}

Suppose that $X$ is a finite indecomposable quandle.  By \cite[Lemmas 2.17 and
2.18]{MR2803792}, $x_i^{|\varphi _i|}=x_j^{|\varphi _j|}$ for all $i,j\in X$.
This implies that the subgroup  $K=\langle x_i^{|\varphi_i|}: i\in X\rangle$ is
cyclic and central.  The \emph{finite enveloping group} is the finite group
$\overline{G_X}=G_X/K$, see \cite[Lemma 2.19]{MR2803792}.  Let $\pi
:G_X\to\overline{G_X}$ be the canonical surjection. 

A quandle $X$ is \emph{injective} if the map $\partial:X\to G_X$, $i\mapsto
x_i$, is injective.  For example, the group $\overline{G_X}$ can be used to
test indecomposable quandles for injectivity.

\begin{lem}
	\label{lem:injective}
	Let $X$ be a finite indecomposable quandle and let $u\in G_X$. Then the
	following hold.
	\begin{enumerate}
    \item The restriction of $\pi$ to the class $u^{G_X}$
      is a quandle isomorphism.
		\item $X$ is injective if and only if
			$X\xrightarrow{\partial}G_X\xrightarrow{\pi}\overline{G_X}$ is injective. 
	\end{enumerate}
\end{lem}

\begin{proof}
	Let $v\in G_{X}$ and assume that $u$ and $v$ are conjugate. Then $u$ and $v$
	have the same $\mathbb{Z}$-degree in $G_{X}$. Now, if $\pi u=\pi v$, then
	$u=vx_{1}^{m|\varphi _1|}$ for some $m\in\mathbb{Z}$.  The
	$\mathbb{Z}$-graduation of $G_{X}$ implies that $m=0$ and hence $u=v$. Thus
	(1) is proved. Now (2) follows from (1).  
\end{proof}

\begin{cor}
	\label{cor:GX_classes}
	Let $X$ be a finite indecomposable quandle and
	\[
		M=\max\{|\mathcal{O}|:\mathcal{O}\text{ is a conjugacy class of }\overline{G_{X}}\}.
	\]
	Then every conjugacy class of $G_{X}$ has at most $M$ elements. 
\end{cor}

\begin{lem} 
	\label{lem:GXisocl}
	Let $X$ be a finite indecomposable quandle. Then $G_X\sim\overline{G_X}$. 
\end{lem}

\begin{proof}
	Let $\pi:G_X\to\overline{G_X}$ be the canonical surjection.  Since the
	elements of $[G_X,G_X]$ have degree zero, $\ker\pi\cap[G_X,G_X]=\langle
  x_1^{|\varphi_1|}\rangle\cap[G_X,G_X]=1$. Then the claim
  follows from Lemma \ref{lem:Hall}\eqref{it:Hall_1} with $G=G_X$,
  $K=\ker \pi$. 
\end{proof}

We conclude the subsection on quandles with two technical lemmas needed for the
proof of our main result.

\begin{lem} 
	\label{lem:invsubset}
  Let $X$ be a crossed set, $Y\subseteq X$ a subset, and
  $$ C(Y)=\{i\in X\,|\,\forall j\in Y: i\trid j=j\}. $$
  Assume that $Y\cup C(Y)=X$. Then $X\trid Y=Y$.
\end{lem}

\begin{proof}
  Let $p,q\in C(Y)$ and $i\in Y$. Then
  $$ (p\trid q)\trid i=p\trid (q\trid i)=p\trid i=i $$
  and hence $C(Y)\trid C(Y)\subseteq C(Y)$. Since $Y\trid C(Y)=C(Y)$
  by definition of $C(Y)$
  and since $X=Y\cup C(Y)$, we conclude that $X\trid C(Y)\subseteq C(Y)$. Since
  $i\tril p=i\tril q=i$, we obtain similarly that $C(Y)\tril C(Y)\subseteq
  C(Y)$ and that $C(Y)\tril X\subseteq C(Y)$. Hence $X\trid C(Y)=C(Y)$ and
  $X\trid (X\setminus C(Y))=X\setminus C(Y)$.

  Let $k\in Y\cap C(Y)$. Then $i\trid k=k$ for all $i\in Y$ since $k\in C(Y)$,
  and $i\trid k=k$ for all $i\in C(Y)$ since $k\in Y$.
  Since $Y\cup C(Y)=X$, we conclude
  that $X\trid \{k\}=\{k\}$ for all $k\in Y\cap C(Y)$.
  This and the first paragraph imply that $X\trid Y=Y$.
\end{proof}

\begin{lem}
	\label{lem:about_quandles}
	Let $X=Y_1\cup Y_2$ be a finite quandle, where $Y_1$ and $Y_2$ are
  disjoint
	$\Inn(X)$-orbits. Assume that there exists an isomorphism of quandles $g:Y_1\to
	Y_2$ and that $Y_1$ is commutative. Then $X$ is isomorphic to the quandle
	structure on $\{1,\dots,2n\}$ given by
	\begin{equation}
		\label{eq:permutations}
		\varphi_i=\begin{cases}
			(n+1\;\cdots\;2n) & \text{if $1\leq i\leq n$,}\\
			(1\;\cdots \;n) & \text{if $n+1\leq i\leq 2n$.}
		\end{cases}
	\end{equation}
\end{lem}

\begin{proof}
	Without loss of generality we may assume that $Y_1=\{1,\dots,n\}$ and
	$Y_2=\{n+1,\dots,2n\}$. For all $i\in Y_1$ and $j\in Y_2$ the permutations
	$\varphi_i$ and $\varphi_j$ commute, since $\supp\varphi_i\subseteq Y_2$ and
	$\supp\varphi_j\subseteq Y_1$. Further,
  $$\varphi_{j\triangleright i}=\varphi_j\varphi_i\varphi_j^{-1}=\varphi_i.$$
  As $\varphi _{k\trid i}=\varphi _i$ for all $k\in Y_1$ by the commutativity
  of $Y_1$, we conclude that $\varphi_i=\varphi_l$ for all $i,l\in Y_1$ since
  $Y_1$ is an $\Inn (X)$-orbit.

	Since $Y_1$ and $Y_2$ are isomorphic, $\varphi_i=\varphi_l$ for all $i,l\in
  Y_2$. Since $Y_1$ is an $\Inn(X)$-orbit, it is a $\varphi _{n+1}$-orbit and
  hence for all $j\in Y_2$ the
	permutation $\varphi_j$ is a cycle of length $|Y_1|$. This implies the
  claim.
\end{proof}
 
\section{Groups with finite-dimensional Nichols algebras}
\label{section:groups}

Here we introduce the groups that realize the examples of decomposable quandles
which are essential for our classification. These quandles are:
\begin{align*}
Z_T^{4,1} &:\;(243)\;(134)\;(142)\;(123)\;\id\\
Z_2^{2,2} &:\;(24)\;(13)\;(24)\; (13)\\
Z_3^{3,1} &:\;(23)\;(13)\;(12)\; \id\\
Z_3^{3,2} &:\;(23)(45)\;(13)(45)\;(12)(45)\;(123)\;(132)\\
Z_4^{4,2} &:\;(24)(56)\;(13)(56)\;(24)(56)\;(13)(56)\;(1234)\;(1432)
\end{align*}

First we study the dimension of group representations.

\begin{lem} 
	\label{lem:dim_asmod_bound}
	\label{lem:dim_asmod_specbound}
	Let $G$ be a group, $x\in G$, and $d\in \N $.  Suppose that $[G:G^x]$ is
	finite.  If $\dim _\K V\le d$ for any finite-dimensional absolutely simple
	$\K G^x$-module $V$, then $\dim _\K U\le d[G:G^x]$ for any finite-dimensional
	absolutely simple $\K G$-module $U$. In particular, $\dim_{\K}U\leq[G:G^x]$
	if $G^x$ is abelian.
\end{lem}

\begin{proof}
	We may assume that $\K $ is algebraically closed.  Let $U$ be a simple $\K
	G$-module with $\dim _\K U<\infty $ and let $V$ be a simple $\K
	G^x$-submodule of $U$.  Then $U=\K G V$ is an epimorphic image of $\K
	G\otimes _{\K G^x} V$ and $\dim _\K V\le d$, and hence $\dim _\K U\le
	d[G:G^x]$.  Now the second claim follows, as finite-dimensional absolutely
	simple modules of abelian groups are one-dimensional.
\end{proof}

\subsection{The group $T$}

Let us consider the group
\[
	T=\langle z\rangle\times\langle x_1,x_2,x_3,x_4\mid
  x_ix_j=x_{\varphi_i(j)}x_i,\quad i,j\in \{1,2,3,4\}\rangle,
\]
where $\{\varphi_i:1\leq i\leq4\}$ is the set of permutations that defines
$(123)^{\Alt_4}$. This group is not nearly abelian, see \cite[Definition
3.1]{MR2732989}, since the commutator subgroup $[T,T]$ is not cyclic. For
example $[x_1,x_2]$ and $[x_1,x_3]$ do not commute.
(One can prove that $[T,T]\simeq Q_8$,
the quaternion group of eight elements.)

\begin{exa}
\label{exa:T+1}
	Let $Z_T^{4,1}=x_1^T\cup z^T$, see Notation~\ref{not:class}(2).  Then $T$ is
	isomorphic to the enveloping group of $Z_T^{4,1}$. 
\end{exa}

\begin{lem}
	\label{lem:quotients_of_T}
	Let $G$ be an epimorphic image of $T$. Then the following hold.
	\begin{enumerate}
		\item\label{it:quotients_of_T}
			$G$ has abelian centralizers.
		\item\label{it:classes_of_T}
			Every conjugacy class of $G$ has at most six elements.
		\item\label{it:degrees_of_T}
			Every finite-dimensional absolutely simple $\K G$-module 
			has dimension at most four.
	\end{enumerate}
\end{lem}

\begin{proof}
  By Lemma~\ref{lem:abelian_centralizers}, we may replace $G$ by a group
  which is isoclinic to $G$. Let $K\vartriangleleft T$ with $G=T/K$. By
  Lemma~\ref{lem:Hall}\eqref{it:Hall_1} we may assume that $K\subseteq [T,T]$.
  Let $X=(123)^{\Alt_4}$. Since $T\sim G_X$ and $G_X\sim \overline{G_X}$ by
  Lemma~\ref{lem:GXisocl}, $T/K\sim \overline{G_X}/L$
  for some $L\vartriangleleft \overline{G_X}$ by
  Lemma~\ref{lem:Hall}\eqref{it:Hall_2}. Now $\overline{G_X}\simeq
  \mathbf{SL}(2,3)$ and the only non-trivial normal subgroups of $\mathbf{SL}(2,3)$
  are its commutator subgroup and its center. Since
  all quotients of $\mathbf{SL}(2,3)$ have abelian
  centralizers, Claim \eqref{it:quotients_of_T} holds.

	To prove (2) we use Corollary \ref{cor:GX_classes}, as every conjugacy class
	of $\mathbf{SL}(2,3)$ has at most six elements. Then Lemma
	\ref{lem:dim_asmod_specbound} and (1) and $|X|=4$ imply (3).
\end{proof}

\subsection{The groups $\Gamma_n$}
\label{subsection:G_n}

Let $n\in\N_{\geq2}$. Recall from \cite{MR2732989} that 
\[
	\Gamma_{n}=\langle g,h,\epsilon\mid hg=\epsilon gh,\, g\epsilon=\epsilon^{-1}g,\, h\epsilon=\epsilon h,\,\epsilon^{n}=1\rangle.
\]
(These groups were denoted by $G_n$ in \cite{MR2732989}.)
Any element of $\Gamma_{n}$ can be written uniquely as $\epsilon^{i}h^{j}g^{k}$,
where $0\leq i\leq n-1$ and $j,k\in\mathbb{Z}$. 
By \cite[\S3]{MR2732989}, the conjugacy classes of $\Gamma_{n}$ are
\begin{gather*}
	{z}^{\Gamma_n}=\{z\},\quad
	{(gz)}^{\Gamma_n}=\{\epsilon^mgz\mid 0\leq m\leq n-1\},\\
	{(h^jz)}^{\Gamma_n}=\{h^jz,\epsilon^{-j}h^jz\},\quad
	{(hgz)}^{\Gamma_n}=\{\epsilon^mhgz\mid 0\leq m\leq n-1\},
\end{gather*}
where $z\in Z(\Gamma_n)=\langle\epsilon^{-1}h^2,h^n,g^2\rangle$ and $1\leq j\leq
n/2$. The centralizers 
	\begin{gather*}
    (\Gamma_n)^{gz}=\langle\epsilon ^{-1}h^2,g,h^n\rangle,\quad
    (\Gamma_n)^{hgz}=\langle\epsilon ^{-1}h^2,hg,h^n\rangle,
		\quad (\Gamma_n)^{h^jz}=\langle\epsilon,h,g^2\rangle
	\end{gather*}
are abelian. The commutator subgroup is $[\Gamma_n,\Gamma_n]=\langle\epsilon\rangle$.

\medskip
Now we show four examples of decomposable quandles.

\begin{exa}
	\label{exa:D4}
	Let $Z_2^{2,2}=h^{\Gamma_2}\cup g^{\Gamma_2}$. Then $Z_2^{2,2}\simeq\D_4$, the dihedral
	quandle of four elements. The enveloping group of $Z_2^{2,2}$ is 
	\[
		\langle x_1,x_2,x_3,x_4\mid x_ix_j=x_{2i-j\imod{4}}x_i,\,i,j\in
    \{1,2,3,4\}\}\simeq \Gamma_2.
	\]
	The isomorphism is given by $x_1\mapsto g$, $x_2\mapsto h$. 
\end{exa}

\begin{exa}
\label{exa:D3+1}
Let $Z_{3}^{3,1}=g^{\Gamma_3}\cup \{\epsilon h\}$. Note that $\epsilon h\in
Z(\Gamma_3)$.
	The enveloping group of $Z_3^{3,1}$ is isomorphic to
	\[
		\langle z\rangle\times\langle x_1,x_2,x_3\mid x_ix_j=x_{2i-j\imod{3}}x_i,
    \,i,j\in \{1,2,3\}\}\simeq \Gamma_3.
	\]
	The latter isomorphism is given by 
	$z\mapsto\epsilon h$, $x_1\mapsto g$, $x_2\mapsto\epsilon g$ and
	$x_3\mapsto\epsilon^2 g$.
\end{exa}

\begin{exa}
	\label{exa:X_3^32}
  Let $Z_{3}^{3,2}=g^{\Gamma_3}\cup h^{\Gamma_3}$. The enveloping group 
	of $Z_3^{3,2}$ is isomorphic to $\Gamma_3$.
\end{exa}

\begin{exa}
  \label{exa:X_4^42}
	Let $Z_4^{4,2}=g^{\Gamma_4}\cup h^{\Gamma_4}$.  The enveloping group of $Z_4^{4,2}$
	is isomorphic to $\Gamma_4$.
\end{exa}

\begin{lem}
	\label{lem:quotients_of_Gn}
	Let $G$ be an epimorphic image of $\Gamma_n$ for some $n\ge 2$. Then the following hold.
	\begin{enumerate}
		\item\label{it:quotients_of_Gn}
			$G$ has abelian centralizers.
		\item\label{it:classes_of_Gn}
			Every conjugacy class of $G$ has at most $n$ elements.
		\item\label{it:degrees_of_Gn}
			Every finite-dimensional absolutely simple $\K G$-module 
			has dimension at most two.
	\end{enumerate}
\end{lem}

\begin{proof}
  Let $p:\Gamma_n\to G$ be the canonical map. If $\epsilon ^k\in \ker p$ for
  some $k>0$, then $G$ is also an epimorphic image of $\Gamma_k$. Since $[\Gamma_n,\Gamma_n]=\langle
  \epsilon \rangle $, we may asume that $\ker p\cap [\Gamma_n,\Gamma_n]=1$.
	Hence $\Gamma_n\sim G$ by Lemma \ref{lem:Hall}\eqref{it:Hall_1}.
  Therefore (1) follows
	from Lemma \ref{lem:abelian_centralizers},
  since $\Gamma_n$ has abelian centralizers. 

	Claim (2) follows from the description of conjugacy classes of $\Gamma_n$.
	Finally (3) follows from Lemma \ref{lem:dim_asmod_specbound} since $h^G$ has two elements. 
\end{proof}

\section{Weyl groupoids of rank two}
\label{section:Weylgroupoid}

Let us consider the map 
\[
\eta:\Z\to\mathbf{SL}(2,\Z),\quad\eta(c)=
\left(\begin{array}{cc}
	c & -1\\
	1 & 0\end{array}\right).
\]
A finite sequence $(c_1,c_2,\dots,c_n)$, $n\in\N$, of positive integers is a
\emph{characteristic sequence} if $\eta(c_1)\cdots\eta(c_n)=-\id$,  and the
entries of the first column of the matrix $\eta(c_1)\cdots\eta(c_i)$ are
non-negative integers for all $i<n$. We denote by $\mathcal{A}^+$ the set of
characteristic sequences. By \cite[Lemma 5.2]{MR2525553},
\begin{equation}
	\label{eq:reduction}
	(c_1,c_2,\dots,c_n)\in\mathcal{A}^+, c_2=1,n\ge 4 \Leftrightarrow
        (c_1-1,c_3-1,c_4,\dots,c_n)\in\mathcal{A}^+.
\end{equation}

\begin{lem}
	\label{lem:sequences}
	Let $(c_1,c_2,\dots,c_n)\in\mathcal{A}^+$. Then $n\ge 3$ and there exists
	$i\in\{1,\dots,n\}$ such that $c_i=1$ and $c_{i+1}\in\{1,2,3\}$, or $c_i=1$
	and $c_{i-1}\in\{1,2,3\}$, where $c_0=c_n$ and $c_{n+1}=c_1$.
\end{lem}

\begin{proof}
	Let $(c_{1},c_{2},\dots,c_{n})\in\mathcal{A}^+$. If $n\le 3$ then $(c_1,\dots
	,c_n)=(1,1,1)$ by \cite[Prop.~5.3(4)]{MR2525553}. Hence we may assume
	that $n>3$.  By \cite[Cor.~4.2]{MR2525553}, there exists $i\in \{1,\dots
	,n\}$ such that $c_i=1$.  Further, $(c_j,c_{j+1},\dots, c_n,c_1,\dots,
	c_{j-1})\in \mathcal{A}^+$ for all $j\in\{1,\dots,n\}$ by
	\cite[Prop.~5.3(2)]{MR2525553}. Also, $c_i=1$ implies that
	$c_{i-1},c_{i+1}>1$ by \eqref{eq:reduction}.  Therefore, without loss of
	generality we may assume that $c=(b_{1},b_{2},\dots,b_{r})$, where
	$b_{i}=(c_{i1},1,c_{i2},1,\dots,c_{i{m_{i}}},1,c_{i{m_i+1}})$ and $c_{ij}\ge
	2$ for all $1\le i\le r$, $1\le j\le m_i+1$, or $c=(d_1,1,\dots ,d_m,1)$ with
	$n=2m$, $d_1,\dots ,d_m\ge 2$.

	Assume first that $c=(b_{1},b_{2},\dots,b_{r})$. Then $m_i\ge 2$ for at
  least one $i$. By applying
	\eqref{eq:reduction} several times we obtain that
	$(b_{1}',b_{2}',\dots,b'_{r})\in \mathcal{A}^+$ , where 
	\[
		b_{i}'=(c_{i{1}}-1,c_{i{2}}-2,c_{i{3}}-2,\dots,c_{i{m_{i}}}-2,c_{i{m_i+1}}-1)
                \quad \text{for all $1\le i\le r$.}
	\]
	Since $(b'_{1},b'_{2},\dots,b'_{r})\in\mathcal{A}^+$,
	by \cite[Cor.~4.2]{MR2525553} there exists
	$i\in\{1,\dots,r\}$ such that $c_{i{1}}-1=1$ or $c_{i{m_i+1}}-1=1$ or
	$c_{ij}-2=1$ for some $j\in\{2,\dots,m_{i}\}$. Then the lemma holds.

	Now assume that $c=(d_{1},1,d_{2},1,\dots,d_{m},1)$, $n=2m$. If $d_i=2$ for
	some $1\le i\le m$ then we are done. Otherwise, after applying
	\eqref{eq:reduction} $m$ times we obtain that
	$(d_{1}-2,d_{2}-2,\dots,d_{m}-2)\in \mathcal{A}^+$. Hence there exists $i\in
	\{1,\dots,m\}$ such that $d_{i}-2=1$. This implies the lemma.
\end{proof}

Cartan schemes of rank two, their Weyl groupoids and their root systems
were studied in
\cite{MR2525553}. An indecomposable Cartan matrix
$C\in \Z ^{2\times 2}$ of finite type is a matrix of the
form \[
	\begin{pmatrix} 
		2 & -c_1 \\ 
		-c_2 & 2
	\end{pmatrix},
\]
where $c_1,c_2\in \N $, $1\le c_1c_2\le 3$.

\begin{cor}
	\label{cor:cartan_type}
	Let $\mathcal{C}=\mathcal{C}(\{1,2\},A,(\rho _i)_{i\in \{1,2\}},(C^a)_{a\in A})$
        be a connected Cartan scheme admitting a finite irreducible root
        system $(R^a)_{a\in A}$.
	Then there exists $a\in A$ such that $C^a\in\Z^{2\times2}$
	is an indecomposable Cartan matrix of finite type.
\end{cor}

\begin{proof}
        Let $a\in A$, $n=|R^a_+|$, $a_1,\dots ,a_{2n}\in A$, $c_1,\dots ,c_{2n}\in \N $
        such that
\begin{align*}
 a_{2r-1}=&(\rho _2\rho _1)^{r-1}(a),& a_{2r}=&\rho _1(\rho _2\rho _1)^{r-1}(a),\\
 c_{2r-1}=&-c_{12}^{a_{2r-1}},& c_{2r}=&-c_{21}^{a_{2r}}
\end{align*}
        for all $r\in \{1,2,\dots,n\}$. Then $(c_1,\dots,c_n)\in \mathcal{A}^+$
	by \cite[Prop.~6.5]{MR2525553}.
	By Lemma \ref{lem:sequences},
	there exists $i$ such that $c_i=1$ and $c_{i+1}\in\{1,2,3\}$, or $c_i=1$ and
	$c_{i-1}\in\{1,2,3\}$. This implies the corollary.
\end{proof}

\section{Nichols algebras over groups} \label{section:Nichols}

Recall that a Yetter-Drinfeld module over a group $G$ is a $\K G$-module
$V=\oplus _{g\in G}V_g$ such that $hV_g\subseteq V_{hgh^{-1}}$ for all $g,h\in
G$.

\begin{lem}
	\label{lem:TandGn}
    Let $G$ be an epimorphic image of one of the groups $\Gamma_2$, $\Gamma_3$,
    $\Gamma_4$ or $T$. Then every finite-dimensional absolutely simple
    Yetter-Drinfeld module over $G$ has dimension at most six.
\end{lem}

\begin{proof}
	Any simple Yetter-Drinfeld module over $G$ is uniquely given by a conjugacy
	class $\mathcal{O}$ of $G$ and an irreducible representation $\rho$ of the
	centralizer of an element of $\mathcal{O}$.  In this case, $\dim
	V=|\mathcal{O}|\mathrm{deg}\,\rho $. Hence the claim follows from Lemmas
	\ref{lem:quotients_of_Gn} and \ref{lem:quotients_of_T}.
\end{proof}

For the study of Nichols algebras over groups the Weyl groupoid of a
tuple of simple Yetter-Drinfeld modules plays an important role.
For the definition we refer to \cite{MR2732989}.

 \begin{thm}
	\label{thm:HS}
	Let $\theta \in \N$, let $H$ be a Hopf algebra with bijective antipode and
	let $M=(M_1,\dots ,M_\theta )$, where each $M_i$ is a simple 
	Yetter-Drinfeld module over $H$.  Assume that $M$ admits all reflections and
	that the Weyl groupoid $\mathcal{W}(M)$ is finite. Then $\NA(M)$ is
	decomposable and admits a finite root system of type $\mathcal{C}(M)$. 
\end{thm}

\begin{proof}
	By \cite[Cor.~2.4]{MR2732989}, $\NA(M)$ is decomposable.  Then the theorem
	becomes precisely \cite[Thm.~2.3]{MR2732989}.
\end{proof}

For any Yetter-Drinfeld module $V$ over a Hopf algebra $H$ with bijective
antipode let $[V]$ denote the isomorphism class of $V$.  The first step in the
proof of Theorem~\ref{thm:main} will be the following claim.

\begin{pro}
	\label{pro:a12=-1}
	Let $H$ be a Hopf algebra with bijective antipode and let
  $M=(M_1,M_2)$ be a pair of simple Yetter-Drinfeld modules over $H$.
  Assume that $M$ admits all reflections and 
  $\mathcal{W}(M)$ is finite. If $(\id -c_{M_2,M_1}c_{M_1,M_2})\not=0$,
  then there exists a pair $N=(N_1,N_2)$
  of simple Yetter-Drinfeld modules over $H$, such that
  $[N]=([N_1],[N_2])\in\mathcal{W}(M)$ and
  $1\le a_{12}^{[N]}a_{21}^{[N]}\le 3$.
\end{pro}

\begin{proof}
  Since $M$ admits all reflections and $(\id -c_{M_2,M_1}c_{M_1,M_2})\not=0$,
  the set of real roots of $\mathcal{W}(M)$ is irreducible. Therefore
  the proposition follows from Corollary~\ref{cor:cartan_type}.
\end{proof}

\begin{thm}
	\label{thm:quandles_and_groups}
	Let $\K$ be a field, $G$ be a non-abelian group, and $V$ and $W$ be two
	Yetter-Drinfeld modules over $G$. Assume that $G$ is generated as a group by
	$\supp(V\oplus W)$, $\supp V$ and $\supp W$ are conjugacy classes of $G$,
	$(\ad V)^2(W)=0$ and $(\ad W)^4(V)=0$.  If $(\id-c_{W,V}c_{V,W})(V\otimes 
	W)\ne0$ then $\supp(V\oplus W)$ is isomorphic to one of the quandles 
	\[
 	 Z_T^{4,1},Z_2^{2,2},Z_3^{3,1},Z_3^{3,2}\text{ and }Z_4^{4,2},
	\]
    and $G$ is isomorphic to an epimorphic image of the corresponding
    enveloping groups $T$, $\Gamma_2$, $\Gamma_3$, $\Gamma_3$ and $\Gamma_4$,
    respectively. 
\end{thm}

Before proving Theorem \ref{thm:quandles_and_groups} we turn our attention to
some consequences.

\begin{thm}
	\label{thm:main}
    Let $\K$ be a field, $G$ be a non-abelian group, and $V$ and $W$ be
    finite-dimensional absolutely simple Yetter-Drinfeld modules over $G$.
    Assume that $G$ is generated by $\supp(V\oplus W)$, the pair $(V,W)$ admits
    all reflections, and the Weyl groupoid of $(V,W)$ is finite.  If
    $(\id-c_{W,V}c_{V,W})(V\otimes	W)\ne0$, then  $G$ is isomorphic to an
    epimorphic image of $\Gamma_n$ for some $n\in\{2,3,4\}$
    or $T$. Moreover,
    $\dim V\leq6$ and $\dim W\leq6$. 
\end{thm}

\begin{proof}
    Proposition \ref{pro:a12=-1} implies that after changing the object of
    $\mathcal{W}(V,W)$, and possibly interchanging $V$ and $W$, we may assume
    that $(\ad V)(W)\not=0$, $(\ad V)^2(W)=0$ and $(\ad W)^4(V)=0$.  Theorem
    \ref{thm:quandles_and_groups} implies that the group $G$ is 
    an epimorphic image of $\Gamma_n$ for $n\in\{2,3,4\}$ or an epimorphic
    image of $T$.  After applying reflections to the pair $(V,W)$ we obtain new
    pairs $(V',W')$ of absolutely simple Yetter-Drinfeld modules over $G$. Thus
    the claim follows from Lemma~\ref{lem:TandGn}.
\end{proof}

\begin{cor} 
	\label{cor:main}
	Let $\K$ be a field, $G$ be a non-abelian group, and $V$ and $W$ be
	finite-dimensional absolutely simple Yetter-Drinfeld modules over $G$.
	Assume that $G$ is generated by $\supp(V\oplus W)$ and $\NA(V\oplus W)$ is
	finite-dimensional.  If $(\id-c_{W,V}c_{V,W})(V\otimes	W)\ne0$, then $\dim
	V\leq6$ and $\dim W\leq6$. 
\end{cor}

\begin{proof}
	Assume that $\NA(V\oplus W)$ is finite-dimensional. Then $(V,W)$ admits all
	reflections by \cite[Cor.~3.18]{MR2766176} and the Weyl groupoid is
	finite by \cite[Prop.~3.23]{MR2766176}. So Theorem \ref{thm:main}
	applies.
\end{proof}

\section{Proof of Theorem \ref{thm:quandles_and_groups}}
\label{section:proof}

The key of our proof is Proposition~\ref{pro:degrees} which
allows us to construct
non-zero elements of $(\ad V)^m(W)$ for any two
Yetter-Drinfeld modules $V,W$ over a
group $G$ and for any $m\in\N$ under some assumption on $G$.  Then we split our
analysis into two parts depending on the question whether $\supp V$ and $\supp
W$ commute.  Finally, we prove Theorem~\ref{thm:quandles_and_groups} in
\S\ref{subsection:proof}. 

In the whole section, let $G$ be a non-abelian group and let $V=\oplus_{s\in
G}V_s$ and $W=\oplus_{t\in G}W_t$ be Yetter-Drinfeld modules over $G$.

\subsection{General considerations}

\begin{lem}
	\label{lem:decomposingG}
	Let $G$ be a group, and $g,h\in G$. Assume that $G$ is
	generated by $g^G$ and $h^G$. Then
	$G=AB$, where $A=\langle \Og\rangle$, $B=\langle\Oh\rangle$, and 
	\[
		AB=\{ab\mid a\in A,\;b\in B\}.
	\]
\end{lem}

\begin{proof}
	Let $r\in\Og$ and $s\in\Oh$. Writing $sr=r(r^{-1}sr)$ we conclude that
	$\Og\Oh=\Oh\Og$. From this the claim follows.
\end{proof}

Recall that $S_n\in \End (V^{\otimes n})$, where $n\in \N $, denotes the
quantum symmetrizer.

\begin{lem}{\cite[Prop.~6.5]{MR2734956}}
	\label{lem:T_n}
	Let $n\in\N$. Then 
	\[
	(\ad V)^n(W)\simeq (S_n\otimes\id)T_n(V^{\otimes n}\otimes W),
	\]
	where $T_n\in\End(V^{\otimes
	n}\otimes W)$ is defined by
	\[
	T_n = (\id-c^2_{n,n+1}c_{n-1,n}\cdots
  c_{1,2})\cdots(\id-c_{n,n+1}^2c_{n-1,n})(\id-c_{n,n+1}^2).
	\]
\end{lem}

\begin{lem}{\cite[Thm.~1.1]{MR2732989}}
	\label{lem:X_n}
	Let $\varphi _0=0$, $X_0^{V,W}=W$, and 
	\begin{align*}
		\varphi_m &= \id-c_{V^{\otimes(m-1)}\otimes W,V}\,c_{V,V^{\otimes(m-1)}
    \otimes W}+(\id\otimes\varphi_{m-1})c_{1,2},\\
		X_m^{V,W} &= \varphi_m(V\otimes X_{m-1})
	\end{align*}
	for all $m\geq1$. Then
	\[
		(S_{n+1}\otimes\id_W)T_{n+1}=\varphi_{n+1}(\id_V\otimes
    S_n\otimes\id_W)(\id_V\otimes T_n).
	\]
	and $(\ad V)^n(W)\simeq X_n^{V,W}$ 
  for all $n\in \N _0$.
\end{lem}

Let $m\in \N _0$.
Recall that an element of $V^{\otimes m}\otimes W$ has \emph{degree}
$(r_1,\dots,r_m,s)$, where
$r_1,\dots,r_m,s\in G$, if it is contained in
$V_{r_1}\otimes\cdots V_{r_m}\otimes W_s$.

Let
$r_1,r_2,\dots,r_m\in\supp V$ and $s\in\supp W$, and write 
\[
	Q_m(r_1,\dots,r_m,s)=(S_m\otimes\id)T_m(V_{r_1}\otimes\cdots\otimes
  V_{r_m}\otimes W_s)\subseteq V^{\otimes m}\otimes W.
\]
Although the vector space $V^{\otimes m}\otimes W$ is graded by $(\supp
V)^m\times \supp W$, the subspace $Q_m(r_1,\dots,r_m,s)$ is usually not graded.
For $t\in V^{\otimes m}\otimes W$ we write $\supp t$ for the set of $d\in
(\supp V)^m\times \supp W$, such that the homogeneous component of $t$ of
degree $d$ is non-zero. We let $$\supp Q=\{\supp t\,|\,t\in Q\}$$ for all
subspaces $Q\subseteq V^{\otimes m}\otimes W$.

\begin{rem} 
	\label{rem:adVmW=0}
  Let $m\in \N $. By Lemma~\ref{lem:T_n},
  $(\ad V)^m(W)=0$ if and only if $\supp Q_m(r_1,\dots ,r_m,s)=0$
  for all $r_1,\dots ,r_m\in \supp V$, $s\in \supp W$.
\end{rem}

\begin{pro}
	\label{pro:degrees}
  Let $m\in\N_0$, $p_1,\dots ,p_m,r_1,\dots,r_m\in\supp V$ and
  $p_{m+1}$, $s\in\supp W$ such that
  \[
		(p_1,\dots,p_m,p_{m+1})\in\supp Q_m(r_1,\dots,r_m,s).
  \]
  Let $p\in \supp V$ and $i\in \{1,\dots,m+1\}$ and assume that
  \begin{gather}
		\label{eq:degreeprop+}
    p_i\trid p\not=p,\quad p_j\triangleright p=p \text{ for all $j$ with
    $i<j\le m+1$,}\\
		\label{eq:degreeprop-}
    p\not\in\{p_j\,|\,1\le j\le m\} \cup
    \{(p_{j+1}p_{j+2}\cdots p_{m+1})^{-1}\triangleright p_j\,|\,
    1\le j<i\}.
	\end{gather}
	Then
	$(p\triangleright p_1,\dots,p\triangleright
  p_{i-1},p,p_i,\dots,p_m,p_{m+1})\in \supp Q_{m+1}(p,r_1,\dots,r_m,s)$.
\end{pro}

\begin{proof}
  Let $t\in Q_m(r_1,\dots ,r_m,s)$ and let $p\in \supp V$.
  By Lemma~\ref{lem:X_n}, $\varphi _{m+1}(v\otimes t)\in
  Q_{m+1}(p,r_1,\dots ,r_m,s)$ for all $v\in V_p$. Moreover,
  $\varphi _{m+1}(v\otimes t)$ is a sum of non-zero homogeneous tensors
	of degrees
  \begin{equation} \label{eq:Qm+1degrees}
	\begin{gathered}
    (p\trid p'_1,\dots ,p\trid p'_{j-1},p,p'_j,\dots ,p'_m,p'_{m+1}),\\
    (p\trid p'_1,\dots ,p\trid p'_{j-1},pp'_j\cdots p'_{m+1}\trid p,
    p\trid p'_j,\dots ,p\trid p'_m,p\trid p'_{m+1})
	\end{gathered}
\end{equation}
  with $1\le j\le m+1$,
  where $(p_1',\dots,p_m',p_{m+1}')\in \supp t$.
  By assumption on $p_1,\dots ,p_m,p_{m+1}$, the tuple
	\begin{equation}
		\label{eq:tuple}
    (p\triangleright p_1,\dots,p\triangleright
    p_{i-1},p,p_i,p_{i+1},\dots,p_m,p_{m+1})
	\end{equation}
  appears among the degrees in \eqref{eq:Qm+1degrees}.
  It suffices to show that it appears exactly once.
  We split the proof into several cases.

	Assume first that \eqref{eq:tuple} is equal to $(p\triangleright
  p_1',\dots,p\triangleright p_{j-1}',p,p_j',\dots,p_m',p'_{m+1})$ for some
	$j\in\{1,\dots,m+1\}$. There are three cases to consider.  First, if $j<i$,
	then $p\triangleright p_j=p$ and hence $p=p_j$, a contradiction to
  \eqref{eq:degreeprop-}. If $j=i$ then we obtain
	$p_l=p_l'$ for all $l\in\{1,\dots,m+1\}$ which gives us just the tuple we
  are looking at. Finally, if $j>i$, then $p=p_{j-1}$, again a contradiction
  to \eqref{eq:degreeprop-}.

	Now assume that \eqref{eq:tuple} is equal to 
  \begin{align} \label{eq:secondtuple}
    (p\trid p'_1,\dots ,p\trid p'_{j-1},pp'_j\cdots p'_{m+1}\trid p,
    p\trid p'_j,\dots ,p\trid p'_m,p\trid p'_{m+1})
  \end{align}
  for some $j\in \{1,\dots ,m+1\}$. Again there are three cases to consider.
	
  If $j>i$ then
  $$p\trid p'_k=p_k \text{ for all $k\in \{j,j+1,\dots ,m+1\}$,}\quad
  pp'_j\cdots p'_{m+1}\trid p=p_{j-1}.$$
  By \eqref{eq:degreeprop+}
  we conclude that $p'_k=p_k$ for all $k\in \{j,j+1,\dots ,m+1\}$ and
  $p=p_{j-1}$.

  If $j=i$ then
  $$p\trid p'_k=p_k \text{ for all $k\in \{i,i+1,\dots ,m+1\}$,}\quad
  pp'_i\cdots p'_{m+1}\trid p=p.$$
  We conclude from \eqref{eq:degreeprop+} that $p'_k=p_k$ for all $k\in
  \{i+1,\dots ,m+1\}$, $p\trid p'_i=p_i$ and $pp'_i\trid p=p$.
  This implies that
  $$p_i\trid p=(p\trid p'_i)\trid p=pp'_ip^{-1}\trid p=pp'_i\trid p=p,$$
  a contradiction to \eqref{eq:degreeprop+}.

  Finally, assume that $1\le j<i$. Then $pp'_j\cdots p'_{m+1}\trid p=p\trid
  p_j$, or
  $$ (p\trid p'_j)(p\trid p'_{j+1})\cdots (p\trid p'_{m+1})\trid p=p\trid
  p_j.$$
  We conclude from this and the equality of \eqref{eq:tuple} and
  \eqref{eq:secondtuple} that
  $$ (p\trid p_{j+1})\cdots (p\trid p_{i-1})pp_i\cdots p_{m+1}\trid p=p\trid
  p_j.$$
  The latter is equivalent to
  $$ pp_{j+1}\cdots p_{i-1}p_i\cdots p_{m+1}\trid p=p\trid p_j,$$
  which after cancelling $p\trid $ gives a contradiction to
  \eqref{eq:degreeprop-}.
\end{proof}

\begin{rem} 
	\label{rem:degrees0}
	If $m=0$ then Proposition~\ref{pro:degrees} reads as follows.
  Let $s\in \supp W$ and
	$p\in \supp V$ and assume that $s\trid p\not=p$. Then $(p,s)\in \supp
	Q_1(p,s)$. 
\end{rem}

\begin{cor} 
	\label{cor:size_suppV}
  Let $m\ge 1$.
  Assume that the following hold.
  \begin{enumerate}
    \item $\supp V$ is indecomposable or $\supp V=Y_1\cup Y_2$ is a decomposition
  into $\Inn (\supp V)$-orbits, and $x\trid Y_1=Y_2$, $x\trid Y_2=Y_1$ for
  all $x\in \supp W$.
    \item $(\ad V)^m(W)\ne 0$, $(\ad V)^{m+1}(W)=0$.
  \end{enumerate}
  Then $|\supp V|\le 2m-1$ if $\supp V$ and $\supp W$ commute,
  and $|\supp V|\le 2m$ otherwise.
\end{cor}

\begin{proof}
  By (2), there exist $r_1,\dots ,r_m,p_1\dots,p_m\in \supp V$,
  $s',p_{m+1}\in \supp W$ and $t\in Q_m(r_1,\dots ,r_m,s')$ such that
  $(p_1,\dots ,p_{m+1})\in \supp t$ and $\varphi _{m+1}(v\otimes t)=0$
  for all $p\in \supp V$, $v\in V_p$. Let
   \begin{align*}
    Y=\{p_j\,|\,1\le j\le m\} \cup
    \{(p_{j+1}p_{j+2}\cdots p_{m+1})^{-1}\triangleright p_j\,|\,
    1\le j\le m\}.
  \end{align*}
  Then $|Y|\le 2m$. Moreover, if $r\trid s=s$ for all $r\in \supp V$, $s\in
  \supp W$, then
  $p_{m+1}^{-1}\trid p_m=p_m$ and hence $|Y|\le 2m-1$.
  Therefore it suffices to prove
  that $Y=\supp V$.

  By Proposition~\ref{pro:degrees}, any $p\in \supp V$ with
  $p\notin Y$ satisfies $p_j\trid p=p$ for all $j\in \{1,\dots ,m+1\}$ and hence
  $\supp V=Y\cup C(Y)$. By Lemma~\ref{lem:invsubset},
  $Y$ is $\Inn (\supp V)$-invariant. Thus $Y=\supp V$ if $\supp V$
  is indecomposable. If $\supp V$ is decomposable as in (1), then
  $p_m$ and $p_{m+1}^{-1}\trid p_m$ are in different components of $\supp V$
  by (1). Therefore the $\Inn (\supp V)$-invariance of $Y$ implies again that
  $Y=\supp V$.
\end{proof}

\begin{cor}
	\label{cor:COMM:degrees}
  Let $r_1,r_2,r_3,r_4\in \supp V$ and $s\in \supp W$. 
	Assume that $\supp V$ and $\supp W$ commute,
	$(r_3,r_4,s)\in\supp Q_2(r_3,r_4,s)$, and 
	\begin{gather}
		r_2\not\in\{r_3,r_4,r_4^{-1}\triangleright r_3\},\quad r_2\triangleright r_4\ne r_4,\\
		r_1\not\in\{r_2\triangleright r_3,r_2,r_4,r_4^{-1}\triangleright r_2,r_4^{-1}\triangleright r_3\},\quad
                r_1\triangleright r_4\ne r_4.
	\end{gather}
	Then $(\ad V)^4(W)\ne0$.
\end{cor}

\begin{proof}
	Let $(p_1,p_2,p_3)=(r_3,r_4,s)\in\supp Q_2(r_3,r_4,s)$. By assumption,
        Conditions \eqref{eq:degreeprop+}, \eqref{eq:degreeprop-}
	with $m=i=2$, $p=r_2$ are fulfilled: $r_4\trid r_2\not=r_2$, $s\trid r_2=r_2$ and
        $r_2\not\in \{r_3,r_4,r_4^{-1}\trid r_3\}$.
	Hence $(r_2\trid r_3,r_2,r_4,s)\in \supp Q_3(r_2,r_3,r_4,s)$ by Proposition \ref{pro:degrees}.

	Now let $(p_1,p_2,p_3,p_4)=(r_2\trid r_3,r_2,r_4,s)\in\supp Q_3(r_2,r_3,r_4,s)$. Then
        $$ (p_3p_4)^{-1}\trid p_2=r_4^{-1}\trid r_2,\quad
           (p_2p_3p_4)^{-1}\trid p_1=r_4^{-1}\trid r_3.
        $$
        By assumption,
        Conditions \eqref{eq:degreeprop+}, \eqref{eq:degreeprop-} with $m=i=3$, $p=r_1$ are fulfilled:
        $p_3\trid p\not=p$, $p_4\trid p=p$, and
	$p\not\in\{p_1,p_2,p_3,(p_3p_4)^{-1}\trid p_2,(p_2p_3p_4)^{-1}\trid p_1\}$.
	Hence $(r_1r_2\trid r_3,r_1\trid r_2,r_1,r_4,s)\in \supp Q_4(r_1,r_2,r_3,r_4,s)$ and therefore
        $(\ad V)^4(W)\not=0$ by Remark~\ref{rem:adVmW=0}.
\end{proof}

\begin{cor}
	\label{cor:COMM:size5}
  Assume that $|\supp V|\ge 5$, $\supp V$ and $\supp W$ commute, $(\ad
  V)(W)\not=0$ and
  $x\trid y\not=y$ for all $x,y\in \supp V$. Then $(\ad V)^4(W)\not=0$.
\end{cor}

\begin{proof}
  Since $(\ad V)(W)\not=0$, there exist $r_4\in \supp V$, $s\in \supp W$
  with $Q_1(r_4,s)\not=0$. Then $(r_4,s)=\supp Q_1(r_4,s)$. Let
  $r_3\in \supp V\setminus \{r_4\}$. Then $r_3\trid r_4\not=r_4$ and hence
  $(r_3,r_4,s)\in \supp Q_2(r_3,r_4,s)$ by Proposition~\ref{pro:degrees}.
  Since $|\supp V|\ge 5$, there exists $r_2\in \supp V$ with $r_2\notin
  \{r_3,r_4,r_4^{-1}\trid r_3,r_4\trid r_3\}$.
  By assumption, $r_2\trid r_4\not=r_4$. By construction,
  $$r_3\notin \{r_2\trid r_3,r_2,r_4,r_4^{-1}\trid r_2,r_4^{-1}\trid r_3\},
  \quad r_3\trid r_4\not=r_4.$$
  Thus Corollary~\ref{cor:COMM:degrees} with $r_1=r_3$ implies that $(\ad
  V)^4(W)\not=0$.
\end{proof}

\begin{cor}
	\label{cor:NC:degrees}
	Let $r_1,r_2,r_3\in \supp V$ and $s\in \supp W$. Assume that the following hold:
\begin{enumerate}
  \item $r_2\trid r_3\not=r_3$,
  \item $r_1\not\in\{r_3r_2\trid r_3,r_3\trid r_2,r_3,s^{-1}\trid r_3,s^{-1}\trid r_2\}$,
  \item $s\trid r_2,s\trid r_3\not\in \{r_2,r_3\}$, 
  \item $r_1\trid s\not=s$ or $r_1\trid r_3\not=r_3$.
\end{enumerate}
Then $(\ad V)^4(W)\ne0$.
\end{cor}

\begin{proof}
	Let $(p_1,p_2)=(r_3,s)$. Then $(p_1,p_2)\in\supp Q_1(r_3,s)$ since $s\trid r_3\not=r_3$.
        Conditions \eqref{eq:degreeprop+} and \eqref{eq:degreeprop-} with $m=1$, $i=2$
        and $p=r_2$ are fulfilled: $p_2\trid p\ne p$, $p=r_2\notin \{r_3,s^{-1}\trid r_3\}$.
	Thus
	\[
 	  (r_2\trid r_3,r_2,s)\in\supp Q_2(r_2,r_3,s)
	\]
        by Proposition \ref{pro:degrees}.

	Let now $(p_1,p_2,p_3)=(r_2\trid r_3,r_2,s)\in\supp Q_2(r_2,r_3,s)$.
        Then Conditions \eqref{eq:degreeprop+} and \eqref{eq:degreeprop-} with $m=2$, $i=3$
        and $p=r_3$ are fulfilled: $s\trid p\ne p$, 
	\[
	r_3\not\in\{p_1,p_2,p_3^{-1}\triangleright p_2,(p_2p_3)^{-1}\triangleright p_1\}
	=\{r_2\trid r_3,r_2,s^{-1}\trid r_2,s^{-1}\trid r_3\}.
	\]
	Hence
	$(r_3r_2\trid r_3,r_3\trid r_2,r_3,s)\in \supp Q_3(r_3,r_2,r_3,s)$
        by Proposition \ref{pro:degrees}.

	Finally, let $(p_1,p_2,p_3,p_4)=(r_3r_2\trid r_3,r_3\trid r_2,r_3,s)\in\supp
  Q_3(r_3,r_2,r_3,s)$
        and let $p=r_1$.
        Then
	\begin{equation*}
                p_4^{-1}\trid p_3=s^{-1}\trid r_3,\,
		(p_3p_4)^{-1}\trid p_2=s^{-1}\trid r_2,\, 
		(p_2p_3p_4)^{-1}\trid p_1=s^{-1}\trid r_3
	\end{equation*}
        and hence $p\notin \{p_1,p_2,p_3,p_4^{-1}\trid p_3,(p_3p_4)^{-1}\trid p_2,(p_2p_3p_4)^{-1}\trid p_1\}$.
        Since $p_4\trid p\not=p$ or $p_3\trid p\not=p$ by (4),
        Proposition~\ref{pro:degrees} with $m=3$ implies that
	$Q_4(r_1,r_3,r_2,r_3,s)\not=0$. Hence $(\ad V)^4(W)\ne 0$ by Lemma \ref{lem:T_n}.
\end{proof}

\begin{cor}
	\label{cor:classification}
	Assume that $\supp V$ is an indecomposable quandle and that $(\ad
	V)(W)\not=0$, $(\ad V)^4(W)=0$.  Then $\supp V$ is isomorphic to one of the
	quandles
	\begin{equation}
		\label{eq:quandles}
  	\{1\},(12)^{\Sym_3},(12)^{\Sym_4},(123)^{\Alt_4},
  	\Aff(5,2),\Aff(5,3),\Aff(5,4),(1234)^{\Sym_4}.
	\end{equation}
\end{cor}

\begin{proof}
	Corollary~\ref{cor:size_suppV} yields that $|\supp V|\le 6$ and Remark
  \ref{rem:uptosix} applies.
\end{proof}

\subsection{Commuting supports}

Let $g,h\in G$.  Assume that $\supp V=\Og $, $\supp W=\Oh $, $G=\langle \Og
\cup \Oh \rangle $ and that $\Og $ and $\Oh $ commute.  We conclude an
implication of $(\ad V)^2(W)=0$, $(\ad W)^4(V)=0$ on $V$ and $W$.

\begin{lem}
	\label{lem:COMM:Oh_indecomposable}
  The quandles $\Og$ and $\Oh $ are indecomposable.
\end{lem}

\begin{proof}
	It is sufficient to prove the claim on $\Oh $.  By Lemma
	\ref{lem:decomposingG} and since $\Og $ and $\Oh $ commute, we obtain that
  $$\Oh =G\trid h=\langle \Oh \rangle \langle \Og \rangle \trid h
  =\langle \Oh \rangle \trid h=\Inn (\Oh )\trid h.$$
  Thus $\Oh $ is indecomposable.
\end{proof}

\begin{lem}
	\label{lem:COMM:adV2W,|Og|=1} 
	Assume that $(\ad V)(W)\not=0$ and $(\ad V)^2(W)=0$. Then $\Og=\{g\}$. 
\end{lem}

\begin{proof}
  This follows from Corollary~\ref{cor:size_suppV} with $m=1$ using
  Lemma~\ref{lem:COMM:Oh_indecomposable}.
\end{proof}

\begin{pro} 
	\label{pro:COMM}
	Assume that $(\ad V)(W)\not=0$, $(\ad V)^2(W)=0$, and $(\ad W)^4(V)=0$.
  Then $\Og\cup\Oh$ is
	isomorphic to $Z_3^{3,1}$ or $Z_T^{4,1}$.
\end{pro}

\begin{proof}
  First, $\Og =\{g\}$ by Lemma~\ref{lem:COMM:adV2W,|Og|=1}.
  Further, $\Oh $ is indecomposable by Lemma~\ref{lem:COMM:Oh_indecomposable}.
  and $|\Oh |\ge 2$ since $G=\langle \Og \cup \Oh
  \rangle $ is non-abelian. Corollary~\ref{cor:size_suppV} implies
  that $|\Oh |\le 5$. Thus, by Corollary~\ref{cor:classification},
  $\Oh $ is isomorphic to one of the quandles $(12)^{\Sym_3}$,
  $(123)^{\Alt_4}$, $\Aff(5,2)$, $\Aff(5,3)$, $\Aff(5,4)$. Assume that $\Oh$ is
  one of the quandles $\Aff(5,2)$, $\Aff(5,3)$, $\Aff(5,4)$.
  Then $|\Oh |=5$ and $x\trid y\not=y$ for any $x,y\in \Oh $ with $x\not=y$.
  Thus $(\ad W)^4(V)\not=0$ by Corollary~\ref{cor:COMM:size5} and the
  proposition follows.
\end{proof}

\subsection{Non-commuting supports}

In this subsection let $g,h\in G$.  Assume that $gh\not=hg$, $\supp V=\Og $,
$\supp W=\Oh $, and $G=\langle \Og \cup \Oh \rangle $.  Then for all $s\in \Oh
$ there exists $r\in \Og $ with $rs\not=sr$. We determine consequences of the
equations $(\ad V)^2(W)=0$ and $(\ad W)^4(V)=0$.

\begin{lem}
	\label{lem:NC:varphi_h}
	Assume that $(\ad V)^2(W)=0$. Then the following hold.
  \begin{enumerate}
    \item $\Og $ is commutative. \label{it:Ogcomm}
    \item $\Og\ne\Oh$. \label{it:OgOh}
    \item $\Og =\langle \Oh \rangle \trid g$. \label{it:Oggeng}
    \item Let $s\in\Oh$. Then there exist $r_1,r_2\in\Og$
      such that $\varphi_s|_{\Og }=(r_1\, r_2)$. \label{it:phih_transp}
    \item $h^2\trid g=g$ and $(gh)^2=(hg)^2$. \label{it:ghgh}
    \item For all $m\in \Z $, $\{x\in \Og \,|\,x\trid (g^m\trid
      h)\not=g^m\trid h\}=\{g,h\trid g\}$. 
  \end{enumerate}
\end{lem}

\begin{proof}
	(4) and (1). First, $|\Og |\ge 2$ and $|\Oh |\ge 2$ since $gh\not=hg$.  Let
	$r_1\in\Og$ and $s\in\Oh$ such that $s\trid r_1\not=r_1$.  Then $(r_1,s)\in
	\supp Q_1(r_1,s)$ by Remark~\ref{rem:degrees0}.  Let $p\in \Og $. Assume
	that $p\notin \{r_1,s^{-1}\trid r_1\}$.  Since $Q_2(p,r_1,s)=0$ because of
	$(\ad V)^2(W)=0$, Proposition~\ref{pro:degrees} implies that $s\trid
	p=p=r_1\trid p$.  Then $\varphi_s|_{\Og }=(r_1\, s^{-1}\trid r_1)$ which is
	the claim in (4).  The equation $r_1\trid p=p$ implies
  that $r_1\trid r_2=r_2$ for all $r_2\in
  \Og $.  Thus \eqref{it:Ogcomm} holds.
  
  (2). If $\Og=\Oh $ then $G=\langle \Og \rangle $ is commutative by
  \eqref{it:Ogcomm}, a contradiction.

  (3). Lemma~\ref{lem:decomposingG} and (1) yield that
  $\Og =G\trid g=\langle \Oh \rangle \langle \Og \rangle \trid g=\langle
  \Oh \rangle \trid g$.

  (5). {}From (1) we know that $g\trid (h\trid g)=h\trid g$ and hence
  $hgh\trid g=h^2\trid g=g$, where the second equation follows from (4).
  This implies (5).

  (6). By (1), $\Og $ is commutative. Thus it
  suffices to prove the claim for $m=0$. The latter follows from
  (4) with $s=h$ since $gh\not=hg$.
\end{proof}

\begin{lem}
	\label{lem:induction}
	Assume that $(\ad V)^2(W)=0$ and that $h$ commutes with $g\trid h$.
  Then the following hold.
  \begin{enumerate}
    \item For all $m\in \Z $, $(h\trid g)\trid (g^m\trid h)=g^{m+1}\trid h$.
    \item $\langle\Og\rangle\triangleright h=\langle g\rangle\triangleright h$.
  \end{enumerate}
\end{lem}

\begin{proof}
	First we prove (1). By Lemma~\ref{lem:NC:varphi_h}\eqref{it:Ogcomm}, $\Og $
	is commutative. Thus it suffices to consider the case $m=0$. Now $(h\trid
	g)\trid h=hg\trid h=g\trid h$ by assumption.

	Now we prove (2). Lemma~\ref{lem:NC:varphi_h}(4) and (1) with $m\in
	\{-1,0\}$, imply that
  \[
  (\Og{})^{\pm 1}\trid h\subseteq \{h\}\cup
  \{g,g^{-1},h\triangleright g,(h\trid g)^{-1}\} \trid h
  \subseteq \{h,g\triangleright h,g^{-1}\triangleright h\}.
	\]
  Now write $\langle\Og\rangle=\cup_{m\in\N_0}A_m$, where $A_m=\{x_1^{\pm1}\cdots
	x_m^{\pm1}\mid x_i\in\Og\}$. It suffices to show that $A_m\trid h\subseteq
  \langle g\rangle \trid h$ for all $m\in \N_0$.
  We proceed by induction on $m$.
  The case $m=0$ is trivial and the case $m=1$ was just proven.
  Let now $m\in \N $ and assume that
  $A_m\trid h\subseteq \langle g\rangle \trid h$. 
	Using the induction hypothesis and the fact that $\Og$ is commutative, see
  Lemma~\ref{lem:NC:varphi_h}(1), we obtain that 
	\begin{align*}
    A_{m+1}\triangleright h &=(\Og){}^{\pm 1}\trid (A_m\trid h)\\
    &\subseteq(\Og)^{\pm 1}\trid (\langle g\rangle \trid h)
    = \langle g\rangle \trid ((\Og)^{\pm 1}\trid h)
    \subseteq \langle g\rangle \trid h.
	\end{align*}
	This implies (2). 
\end{proof}

\begin{lem}
	\label{lem:OgUOh=D4}
	Assume that $(\ad V)^2(W)=0$ and that $\Oh$ is commutative.
	Then $\Og\cup\Oh$ is isomorphic to $Z_2^{2,2}$.
\end{lem}

\begin{proof}
	Lemma~\ref{lem:NC:varphi_h}\eqref{it:phih_transp} implies that $ghg\trid h=h$.  
	Since $\Oh$ is commutative, $hg\triangleright h=g\triangleright h$ and hence
	$h=ghg\triangleright h=g^2\triangleright h$.  Therefore
	$$ \Oh= \langle \Og\rangle \langle \Oh \rangle \trid h
  =\langle \Og \rangle \trid h=\{h,g\triangleright h\}$$
  by Lemmas \ref{lem:decomposingG} and \ref{lem:induction}(2)
  and since $\Oh $ is commutative.
	Recall that $\Og $ is commutative by Lemma~\ref{lem:NC:varphi_h}\eqref{it:Ogcomm}
	and that $h^2\trid g=g$ by Lemma \ref{lem:NC:varphi_h}\eqref{it:ghgh}.
  From Lemma \ref{lem:decomposingG} we obtain 
	\[
		\Og=\langle\Oh\rangle\langle\Og\rangle\triangleright g
		=\langle\Oh\rangle\triangleright g
		=\langle h,g\trid h\rangle\trid g\subseteq \langle g,h\rangle \trid g=\{h\triangleright g,g\}.
	\]
	Therefore $\Og =\{h\trid g,g\}$ and $\Og\cup\Oh\simeq Z_2^{2,2}$ as quandles.
\end{proof}

\begin{lem}
	\label{lem:commuting_elements}
	Let $x,y\in\Oh$ such that $x\trid y=y$. Assume that
	$y\trid z\not=z$ for all $z\in \Oh\setminus\{x,y\}$,
	and that $\varphi_x|_{\Og}=(r\,s)$ for some $r,s\in\Og$, $r\not=s$.
  Then $\varphi_y|_{\Og}=(r\,s)$.
\end{lem}

\begin{proof}
  Since $x,y\in \Oh $ and $\varphi _x|_{\Og }=(r\,s)$,
  there exist $a,b\in \Og $ such that $\varphi _y|_{\Og
  }=(a\,b)$. Assume that $(a\,b)\not=(r\,s)$. Then $|\{r,s,a,b\}|=4$
  since $\varphi _x|_{\Og }$ and $\varphi_y|_{\Og}$
  commute. Let
	$z=r\triangleright x$.
  First, $z=r\trid x\not=x$ since $x\trid r\not=r$.
  Second, $r\trid x\not=y$ since $\varphi _z|_{\Og }=(r\,r\trid s)\not=(a\,b)$.
	Hence $y\trid z\not=z$ by assumption, a contradiction to $y\trid (r\trid
  x)=(y\trid r)\trid (y\trid x)=r\trid x$. 
\end{proof}

\begin{lem}
	\label{lem:NC:2generated}
  Let $x,y\in \Oh $, $\psi _x=\varphi _x|_{\Og }$ and $\psi _y=\varphi
  _y|_{\Og }$. Assume that $(\ad V)^2(W)=0$ and that $x,y$ generate the
  quandle $\Oh$. If $\psi _x=\psi _y$ then $|\Og |=2$. Otherwise
  $|\Og |=3$ and $\psi _x \psi _y\not=\psi _y\psi _x$.
\end{lem}

\begin{proof}
  By Lemma~\ref{lem:NC:varphi_h}\eqref{it:Ogcomm} and \eqref{it:phih_transp},
  $\Og$ is commutative and there
	exist $g_1,g_2\in\Og$ such that
	$\psi_x=(g_1\;x\triangleright g_1)$ and
	$\psi_y=(g_2\;y\triangleright g_2)$.
	Assume now that $|\{g_1,x\triangleright g_1,g_2,y\triangleright g_2\}|=4$.
	Then $|\Og |\ge 4$.
  On the other hand, Lemma~\ref{lem:decomposingG} and the commutativity of
  $\Og$ imply that
  \begin{align} \label{eq:Og}
	  \Og=G\triangleright g_1
    =\langle \Oh \rangle \langle\Og\rangle\triangleright g_1
    =\langle\Oh\rangle\triangleright g_1
	  =\langle x,y\rangle \trid g_1=\{g_1,x\triangleright g_1\},
  \end{align}
  a contradiction to $|\Og |\ge 4$. Hence
	$|\{g_1,x\triangleright g_1,g_2,y\triangleright g_2\}|\le 3$
  and the lemma
  follows by two calculations similar to \eqref{eq:Og}.
\end{proof}

\begin{lem}
	\label{lem:NC:decomposing_Oh}
	Assume that $(\ad V)^2(W)=0$, $\Oh$ is decomposable and let
	$\Oh=Y_1\cup\dots\cup Y_k$ be the decomposition of $\Oh$ into orbits of
	the inner group of $\Oh$.  Then $k=2$ and $x\triangleright Y_1=Y_2$,
	$x\triangleright Y_2=Y_1$ for all $x\in \Og $.
\end{lem}

\begin{proof}
  First, $\Oh=g\triangleright\Oh=(g\triangleright
		Y_1)\cup\cdots\cup(g\triangleright Y_k)$
  is a decomposition into $\Inn (\Oh )$-orbits:
  $(g\trid y)\trid (g\trid Y_i)=g\trid (y\trid Y_i)=g\trid Y_i$ for all $y\in
  \Oh $, $1\le i\le k$.
	Thus $\varphi_g$ permutes the orbits $Y_1,\dots,Y_k$.
  Since $\Og =\langle \Oh \rangle \trid g$ by
  Lemma~\ref{lem:NC:varphi_h}\eqref{it:Oggeng},
  each $x\in \Og$ permutes the $\Inn (\Oh )$-orbits $Y_1,\dots ,Y_k$ in the
  same way as $g$ does.
  Let $Y\subset \Oh $ be the $\Inn (\Oh )$-orbit of $h$.  
  As $G$ is generated by $\Og \cup \Oh $ and $\Oh $ is a
  conjugacy class of $G$, we conclude that
	\[
		\Oh=Y\cup(g\triangleright Y)\cup\cdots\cup(g^{k-1}\triangleright Y).
	\]
  By Lemma \ref{lem:NC:varphi_h}\eqref{it:ghgh}, $ghg\triangleright h=h$ and
	hence $h\in Y\cap(g^2\triangleright Y)$. Thus $g^2\triangleright Y=Y$ and
	hence $\Oh=Y\cup(g\triangleright Y)$ since $\Oh $ is decomposable.
\end{proof}

\begin{lem}
	\label{lem:NC:Oh_is_decomposable}
	Assume that $(\ad V)^2(W)=0$, $(\ad W)^4(V)=0$ and that $\Oh$ is
	decomposable. Then $\Og\cup\Oh$ is isomorphic to $Z_2^{2,2}$ or to
	$Z_4^{4,2}$.
\end{lem}

\begin{proof}
	By Lemma \ref{lem:NC:decomposing_Oh}, $\Oh=Y_1\cup Y_2$, where $Y_1$ and
	$Y_2$ are the $\Inn(\Oh)$-orbits of $\Oh $. Moreover, $x\trid Y_1=Y_2$ and
	$x\trid Y_2=Y_1$ for all $x\in \Og $. Thus Corollary~\ref{cor:size_suppV}
	implies that $|\Oh|\leq6$. There are two cases to consider. 

	Assume first that $Y_1$ is non-commutative. Then $Y_1\simeq
	Y_2\simeq(12)^{\Sym_3}$ by Remark~\ref{rem:smallcrossedsets}.  Let $r_3\in
	Y_1$, $r_2\in Y_1\setminus\{r_3\}$ and $r_1\in
	Y_2\setminus\{g^{-1}\triangleright r_3,g^{-1}\triangleright r_2\}$.  By
	Corollary \ref{cor:NC:degrees} with $s=g$, $(\ad W)^4(V)\ne0$, a
	contradiction.

	Assume now that $Y_1$ is commutative. By Lemma \ref{lem:about_quandles},
	the permutations $\varphi_i$ defining $\Oh$ are given by
	\eqref{eq:permutations}. Further, $x\trid y\not=y$ and hence $y\trid
  x\not=x$ for all $x\in \Og $, $y\in \Oh $.
  But $\varphi _y|_{\Og }$ is a transposition for all $y\in \Oh $ by
  Lemma~\ref{lem:NC:varphi_h}\eqref{it:phih_transp}, and hence $|\Og |=2$.
  
  If $|Y_1|=1$ then $\Oh $ is commutative and
  $\Og\cup\Oh\simeq Z_2^{2,2}$ by Lemma~\ref{lem:OgUOh=D4}.
	Suppose next that $|Y_1|=2$. Then $\Oh\simeq Z_2^{2,2}$ by
  Lemma~\ref{lem:about_quandles}. Let $h'\in Y_1$ with $h'\not=h$.
  Since $ghg\trid h=h$ by Lemma~\ref{lem:NC:varphi_h}\eqref{it:ghgh}
  and since $hg\trid h\not=g\trid h$, we conclude that $g^2\trid h\not=h$,
  $\varphi _g=(h\,g\trid h\,h'\,g\trid h')$, and
  $\varphi _{h\trid g}=\varphi _g^{-1}$.
  Therefore $\Og \cup \Oh \simeq Z_4^{4,2}$.

	Finally, assume that $|Y_1|=3$. Let $r_2\in Y_1$.
  Then $r_2\triangleright x\ne x$
	for all $x\in Y_2$, by Lemma \ref{lem:about_quandles}. Now take $r_3\in
	Y_2\setminus\{g\triangleright r_2,g^{-1}\triangleright r_2\}$ and
	$r_1\in Y_1\setminus\{r_3\triangleright r_2,g^{-1}\triangleright r_3\}$.
  Then $(\ad W)^4(V)\ne0$ by Corollary \ref{cor:NC:degrees} with $s=g$,
  a contradiction.
\end{proof}

\begin{lem}
	\label{lem:123^A4}
	Assume that $h^2\trid g=g$. Then $h^2\trid (g\trid h)=g\trid h$.  In
	particular, $\Oh$ is not isomorphic to any of $(123)^{\Alt_4}$,
  $\Aff (5,2)$ and $\Aff (5,3)$.
\end{lem}

\begin{proof}
  The first claim follows from the definition of a quandle.
  Since $h$ and $g\trid h$ are fixed points of $\varphi _h^2|_{\Oh }$,
  the second claim follows from Remark~\ref{rem:uptosix}.
\end{proof}

\begin{lem}
	\label{lem:NC:affine5}
	Assume that $(\ad V)^2(W)=0$ and $(\ad W)^4(V)=0$. Then
	$\Oh$ is not isomorphic to $\Aff(5,4)$.
\end{lem}

\begin{proof}
	Assume that $\Oh\simeq\Aff(5,4)$.  Then
	$\Oh$ can be generated by two elements $x$ and $y$, $x\ne y$.
  By Lemma~\ref{lem:NC:varphi_h}\eqref{it:phih_transp}, $\varphi _x|_{\Og }$
  and $\varphi _y|_{\Og }$ are transpositions.
	By Lemma \ref{lem:NC:2generated}, either $|\Og |=2$ or $|\Og |=3$, $\varphi
  _x|_{\Og }\not=\varphi _y|_{\Og }$.
	Assume the second case. Let $z\in \Og $ such that $x\trid z\not=z$, $y\trid
  z\not=z$. Then $x\trid z\not=y\trid z$,
  $ x\trid (y\trid z)=y\trid z$, $y\trid (x\trid z)=x\trid z$. Therefore
  $$ xyxyx\trid z=y\trid z\not=x\trid z=yxyxy\trid z,$$
  a contradiction to $xyxyx=yxyxy$ in $G$.
  Hence $|\Og |=2$.

	Now $g\trid z\not=z$ for all $z\in \Oh $ and therefore we
	may assume that $g^3\triangleright h\ne h$. Moreover, for all $z_1,z_2\in
  \Oh $ there exists $z\in \Oh $ such that $z\trid z_1=z_2$. So let
  $r_2\in\Oh$ such that $r_2\triangleright h=g\triangleright h$ and let $r_3=h$.
  Since $ghgh=hghg$ by Lemma~\ref{lem:NC:varphi_h}\eqref{it:ghgh},
  we conclude that $r_2\trid r_3\ne r_3$, $g\trid r_2\not=r_2$,
  $g\trid r_2\not=r_3$ since
  $$ (g\trid r_2)\trid (g\trid h)=g\trid (r_2\trid h)=g^2\trid h\ne
  g^{-1}\trid h=h\trid (g\trid h),$$
  $g\trid r_3\not\in\{r_2,r_3\}$. Moreover, $r_3r_2\trid r_3=hg\trid
  h=g^{-1}\trid r_3$, and hence there exists $r_1\in \Oh \setminus
  \{r_3r_2\trid r_3,r_3\trid r_2,r_3,g^{-1}\trid r_2,g^{-1}\trid r_3\}$.
  Since $r_1\trid g\not=g$, Corollary~\ref{cor:NC:degrees} with $s=g$
  implies that $(\ad W)^4(V)\ne 0$. This is a contradiction and hence
	$\Oh\not\simeq\Aff(5,4)$.
\end{proof}

\begin{lem}
	\label{lem:12^S4}
	Assume that $(\ad V)^2(W)=0$ and $(\ad W)^4(V)=0$. Then $\Oh$ is neither
	isomorphic to $(1234)^{\Sym_4}$ nor to $(12)^{\Sym_4}$.
\end{lem}

\begin{proof}
  Assume that $\Oh \simeq (1234)^{\Sym_4}$ or $\Oh \simeq (12)^{\Sym_4}$.
  Let $r_3\in \Oh $, $s\in \Og $ with $s\trid r_3\not=r_3$
  and let $x\in \Oh \setminus \{r_3\}$ with $r_3\trid x=x$. It suffices to
  show that $s\trid r_3=x$, $s\trid x=r_3$, and
  $\varphi _s|_{\Oh }\not=(x\,r_3)$.
  Indeed, let $r_2\in \Oh \setminus \{r_3,x\}$ with $s\trid r_2\not=r_2$ and
  let $r_1\in \Oh \setminus \{r_3r_2\trid r_3,r_3\trid r_2,r_3,
  s^{-1}\trid r_3,s^{-1}\trid r_2\}$. Then $r_1\trid r_3\not=r_3$
  since $r_1\not=r_3$ and $r_1\not=s^{-1}\trid r_3=x$, and hence
  Corollary~\ref{cor:NC:degrees} contradicts to $(\ad W)^4(V)=0$.

  Now we show that $s\trid r_3=x$, $s\trid x=r_3$.
  First, $\varphi _{r_3}|_{\Og }$ and $\varphi _x|_{\Og }$ are transpositions
  by Lemma~\ref{lem:NC:varphi_h}\eqref{it:phih_transp}.
  If $\Oh \simeq (1234)^{\Sym _4}$ then $r_3^2\trid (s\trid r_3)=s\trid r_3$
  and $\varphi _{r_3}^2|_{\Oh }$ has only $r_3$ and $x$ as fixed points.
  Hence $s\trid r_3=x$ and similarly $s\trid x=r_3$.
  If $\Oh \simeq (12)^{\Sym _4}$ then
  Lemma~\ref{lem:commuting_elements} implies that
  $\varphi _{r_3}|_{\Og }=\varphi _x|_{\Og }$.
  Hence $r_3x\trid (s\trid r_3)=s\trid r_3$.
  Since $\varphi _{r_3}\varphi _x|_{\Oh}$ has only $r_3$ and $x$ as fixed
  points, we conclude that $s\trid r_3=x$ and similarly $s\trid x=r_3$.

	Now we show that there exists $y\in \Oh \setminus \{r_3,x\}$ such that
	$s\trid y\not=y$. If $\Oh \simeq (1234)^{\Sym_4}$ then
	Lemma~\ref{lem:NC:2generated} implies that $|\Og |\le 3$ and the claim holds.
	If $\Oh \simeq (12)^{\Sym_4}$ then let $z\in \Oh \setminus \{r_3,x\}$. Then
	$r_3,x$ and $z$ generate $\Oh $ as a quandle.  Recall that $\varphi
  _{r_3}|_{\Og }=\varphi _x|_{\Og }=(s\,x\trid s)$. If $\varphi _z|_{\Og
  }=(a\,b)$ with $|\{s,x\trid s,a,b\}|=4$ then $|\Og |=2$
  by a calculation similar to
	\eqref{eq:Og} of Lemma \ref{lem:NC:2generated}, a contradiction.
  Otherwise $|\Og |\le 3$ as in the proof of
  Lemma~\ref{lem:NC:2generated}. Then again $y\trid s\not=s$ for four or six
  elements $y\in \Oh $.
\end{proof}

\begin{pro}
	\label{pro:NC}
	Assume that $(\ad V)^2(W)=0$
	and $(\ad W)^4(V)=0$.  Then $\Og\cup\Oh$ is isomorphic to $Z_2^{2,2}$, $Z_3^{3,2}$ or
	$Z_4^{4,2}$.
\end{pro}

\begin{proof}
	First, $\Og\not=\Oh $ by Lemma~\ref{lem:NC:varphi_h}\eqref{it:OgOh}.
  If $\Oh $ is
	commutative then $g^G\cup h^G\simeq Z_2^{2,2}$ by Lemma \ref{lem:OgUOh=D4}.  If $\Oh$
	is decomposable then $g^G\cup h^G\simeq Z_2^{2,2}$ or $g^G\cup h^G\simeq
	Z_4^{4,2}$ by Lemma~\ref{lem:NC:Oh_is_decomposable}. Finally, suppose that
	$h^G$ is non-commutative and indecomposable. Then
	Corollary~\ref{cor:classification} implies that $\Oh$ is isomorphic to one of
	the non-commutative quandles of \eqref{eq:quandles}.
  Since $h^2\trid g=g$ by Lemma \ref{lem:NC:varphi_h}\eqref{it:ghgh},
  Lemmas \ref{lem:123^A4},
	\ref{lem:NC:affine5}, and \ref{lem:12^S4} imply that
	$\Oh\simeq(12)^{\Sym_3}$. Then $|\Og |=2$ or $|\Og |=3$ by Lemma
	\ref{lem:NC:2generated} and $\Og $ is commutative by
  Lemma~\ref{lem:NC:varphi_h}\eqref{it:Ogcomm}. If $|\Og |=2$
  then $g\trid x\not=x$ for all $x\in \Oh $ and hence $\varphi _g$
  is a three-cycle and $\varphi _{h\trid
	g}=\varphi _h\varphi _g\varphi _h^{-1}=\varphi _g^{-1}$.
  Thus $g^G\cup h^G\simeq Z_3^{3,2}$.  If $|\Og |=3$
	then $(g\trid h)\trid g=g\trid (h\trid g)=h\trid g$. Then
	Lemma~\ref{lem:NC:varphi_h}\eqref{it:phih_transp} implies that $\varphi
	_{g\trid h}|_{\Og }=\varphi _h|_{\Og }=(g\,h\trid g)$, a contradiction to
	Lemma~\ref{lem:NC:2generated} and $|\Og |=3$.
\end{proof}

\subsection{The proof of Theorem \ref{thm:quandles_and_groups}} 
\label{subsection:proof}

	Let $g\in \supp V$, $h\in \supp W$. Then $\supp V=\Og $, $\supp W=\Oh $ by
	assumption.  Let $X=\Og\cup\Oh$.  If $\Og $ and $\Oh $ commute then $X\simeq
	Z_3^{3,1}$ or $X\simeq Z_T^{4,1}$ by Proposition \ref{pro:COMM}.  Otherwise,
	$\Og$ and $\Oh$ do not commute and $X\simeq Z_2^{2,2}$ or $X\simeq Z_4^{4,2}$
	or $X\simeq Z_3^{3,2}$ by Proposition \ref{pro:NC}.
	The enveloping groups of the quandles $Z_T^{4,1}$, $Z_2^{2,2}$, $Z_3^{3,1}$,
	$Z_3^{3,2}$ and $Z_4^{4,2}$ were computed in \S\ref{section:groups}. Hence
	the theorem follows from the universal property of the enveloping group, see
	Remark \ref{rem:universal_property}.  \qed

\def\cprime{$'$}


\begin{thebibliography}{10}

\bibitem{MR2786171}
N.~Andruskiewitsch, F.~Fantino, M.~Gra{\~n}a, and L.~Vendramin.
\newblock Finite-dimensional pointed {H}opf algebras with alternating groups
  are trivial.
\newblock {\em Ann. Mat. Pura Appl. (4)}, 190(2):225--245, 2011.

\bibitem{MR2745542}
N.~Andruskiewitsch, F.~Fantino, M.~Gra{\~n}a, and L.~Vendramin.
\newblock Pointed {H}opf algebras over the sporadic simple groups.
\newblock {\em J. Algebra}, 325:305--320, 2011.

\bibitem{MR1994219}
N.~Andruskiewitsch and M.~Gra{\~n}a.
\newblock From racks to pointed {H}opf algebras.
\newblock {\em Adv. Math.}, 178(2):177--243, 2003.

\bibitem{MR2766176}
N.~Andruskiewitsch, I.~Heckenberger, and H.-J. Schneider.
\newblock The {N}ichols algebra of a semisimple {Y}etter-{D}rinfeld module.
\newblock {\em Amer. J. Math.}, 132(6):1493--1547, 2010.

\bibitem{MR1659895}
N.~Andruskiewitsch and H.-J. Schneider.
\newblock Lifting of quantum linear spaces and pointed {H}opf algebras of order
  {$p\sp 3$}.
\newblock {\em J. Algebra}, 209(2):658--691, 1998.

\bibitem{MR1780094}
N.~Andruskiewitsch and H.-J. Schneider.
\newblock Finite quantum groups and {C}artan matrices.
\newblock {\em Adv. Math.}, 154(1):1--45, 2000.

\bibitem{MR2630042}
N.~Andruskiewitsch and H.-J. Schneider.
\newblock On the classification of finite-dimensional pointed {H}opf algebras.
\newblock {\em Ann. of Math. (2)}, 171(1):375--417, 2010.

\bibitem{Ang}
I.~{Angiono}.
\newblock {A presentation by generators and relations of {N}ichols algebras of
  diagonal type and convex orders on root systems}.
\newblock {\em Accepted for publication in J. Europ. Math. Soc.,
  arXiv:1008.4144}.

\bibitem{Ang_AS}
I.~{Angiono}.
\newblock On {N}ichols algebras of diagonal type.
\newblock {\em J. Reine Angew. Math.}, 683:189--251, 2013.

\bibitem{MR2525553}
M.~Cuntz and I.~Heckenberger.
\newblock Weyl groupoids of rank two and continued fractions.
\newblock {\em Algebra Number Theory}, 3(3):317--340, 2009.

\bibitem{MR2803792}
M.~Gra{\~n}a, I.~Heckenberger, and L.~Vendramin.
\newblock Nichols algebras of group type with many quadratic relations.
\newblock {\em Adv. Math.}, 227(5):1956--1989, 2011.

\bibitem{MR0003389}
P.~Hall.
\newblock The classification of prime-power groups.
\newblock {\em J. Reine Angew. Math.}, 182:130--141, 1940.

\bibitem{MR2207786}
I.~Heckenberger.
\newblock The {W}eyl groupoid of a {N}ichols algebra of diagonal type.
\newblock {\em Invent. Math.}, 164(1):175--188, 2006.

\bibitem{MR2462836}
I.~Heckenberger.
\newblock Classification of arithmetic root systems.
\newblock {\em Adv. Math.}, 220(1):59--124, 2009.

\bibitem{MR2732989}
I.~Heckenberger and H.-J. Schneider.
\newblock Nichols algebras over groups with finite root system of rank two {I}.
\newblock {\em J. Algebra}, 324(11):3090--3114, 2010.

\bibitem{MR2734956}
I.~Heckenberger and H.-J. Schneider.
\newblock Root systems and {W}eyl groupoids for {N}ichols algebras.
\newblock {\em Proc. Lond. Math. Soc. (3)}, 101(3):623--654, 2010.

\bibitem{rank2}
I.~{Heckenberger} and L.~{Vendramin}.
\newblock The classification of {N}ichols algebras with finite root system of
  rank two.
\newblock {\em arXiv:1311.2881}, 2013.

\bibitem{examples}
I.~{Heckenberger} and L.~{Vendramin}.
\newblock New examples of {N}ichols algebras with finite root system of rank
  two.
\newblock {\em arXiv:1309.4634}, 2013.

\bibitem{MR2390080}
I.~Heckenberger and H.~Yamane.
\newblock A generalization of {C}oxeter groups, root systems, and {M}atsumoto's
  theorem.
\newblock {\em Math. Z.}, 259(2):255--276, 2008.

\bibitem{MR1763385}
V.~K. Kharchenko.
\newblock A quantum analogue of the {P}oincar\'e-{B}irkhoff-{W}itt theorem.
\newblock {\em Algebra Log.}, 38(4):476--507, 509, 1999.

\bibitem{MR2759715}
G.~Lusztig.
\newblock {\em Introduction to quantum groups}.
\newblock Modern Birkh\"auser Classics. Birkh\"auser/Springer, New York, 2010.
\newblock Reprint of the 1994 edition.

\bibitem{MR0360813}
D.~M. Rocke.
\newblock {$p$}-groups with abelian centralizers.
\newblock {\em Proc. London Math. Soc. (3)}, 30:55--75, 1975.

\bibitem{MR1632802}
M.~Rosso.
\newblock Quantum groups and quantum shuffles.
\newblock {\em Invent. Math.}, 133(2):399--416, 1998.

\bibitem{MR2926571}
L.~Vendramin.
\newblock On the classification of quandles of low order.
\newblock {\em J. Knot Theory Ramifications}, 21(9):1250088, 10, 2012.

\end{thebibliography}
\end{document}